\documentclass[12pt,a4paper]{amsproc}
\usepackage{amsmath}
\usepackage{hyperref}
\usepackage{cite}

\DeclareMathOperator{\RE}{Re} 
\numberwithin{equation}{section}
\newtheorem{theorem}{Theorem}[section]
\newtheorem{lemma}{Lemma}[section]

\theoremstyle{remark}
\newtheorem{remark}{Remark}[section]

\makeatletter

\allowdisplaybreaks

\numberwithin{equation}{section} \numberwithin{theorem}{section}

\title[Radius Constants for Functions with Fixed Second Coefficient]{Radius Constants for Analytic Functions\\ with Fixed Second Coefficient}

\author[R. M. Ali]{Rosihan M. Ali}
\address{School of Mathematical Sciences, Universiti Sains
Malaysia, 11800 Penang, Malaysia} \email{rosihan@cs.usm.my}

\author[M. M. Nargesi]{Mahnaz M. Nargesi }
\address{School of Mathematical Sciences, Universiti Sains
Malaysia, 11800 Penang Malaysia} \email{moradinargesik@yahoo.com}

\author{V. Ravichandran}
\address{Department of Mathematics,
University of Delhi, Delhi 110 007, India;
School of Mathematical Sciences, Universiti Sains
Malaysia, 11800 Penang Malaysia}
\email{vravi@maths.du.ac.in}



\thanks{The work presented here was supported in parts by a Research University grant  from Universiti Sains Malaysia.}

\begin{document}
\begin{abstract}
Let  $f(z)=z+\sum_{n=2}^{\infty}a_nz^n$ be analytic in the unit disk with second coefficient $a_2$ satisfying $|a_2|=2b$, $0\leq b\leq1$. Sharp radius of Janowski starlikeness and other radius constants are obtained when $|a_n|\leq cn+d$ ($c,d\geq0$) or $|a_n|\leq c/n$ ($c>0$) for $n\geq3$.

\end{abstract}
\maketitle

\section{Introduction}
Let $\mathcal{A}$ denote the class of  analytic functions $f$ defined in $\mathbb{D}=\left\{z\in \mathbb{C}:|z| <1\right\}$, normalized by $f(0)=0=f'(0)-1$, and let $\mathcal{S}$ denote its subclass  consisting of univalent functions. If $f(z)=z+\sum_{n=2}^{\infty}a_nz^n\in\mathcal{S}$,  de Branges \cite{Branges} proved the Bieberbach conjecture that $|a_n|\leq n\  (n\geq2)$. However, the inequality $|a_n|\leq n\ (n\geq 2)$ does not  imply $f$ is  univalent; for example, $f(z)=z+2z^2$ is not a member of $\mathcal{S}$.

Gavrilov \cite{Gavrilov} showed that the radius of univalence for functions $f\in\mathcal{A}$ satisfying $|a_n|\leq n\ (n\geq2$)
 is the real root $r_0\simeq 0.164$ of the equation $2(1-r)^3-(1+r)=0$, and the result is sharp for $f(z)=2z-z/(1-z)^2$. Gavrilov  also proved that the radius of univalence for functions $f\in\mathcal{A}$ satisfying $|a_n|\leq M\ (n\geq 2)$ is $1-\sqrt{M/(1+M)}$. The inequality $|a_n|\leq M$ holds for functions  $f\in\mathcal{A}$ satisfying $|f(z)|\leq M$, and for these functions,  Landau \cite{Valiron} proved that the radius of univalence   is $M-\sqrt{M^2-1}$.
 Yamashita \cite{Yamashita} showed that the radius of univalence obtained by Gavrilov  \cite{Gavrilov} is  also the radius of starlikeness  for functions $f\in\mathcal{A}$ satisfying $|a_n|\leq n$ or $|a_n|\leq M\ (n\geq 2)$. Additionally Yamashita \cite{Yamashita} determined that the radius of convexity for  functions $f\in\mathcal{A}$ satisfying $|a_n|\leq n\ (n\geq2$) is the real root $r_0\simeq 0.090$ of the equation  $2(1-r)^4-(1+4r+r^2)=0$, while the radius of convexity for  functions $f\in\mathcal{A}$ satisfying $|a_n|\leq M\ (n\geq 2)$ is the real root of
    \[\left(M+1\right)(1-r)^3-M(1+r)=0.\]
 Recently Kalaj \emph{et al.} \cite{ponnusamy} obtained the radii of univalence, starlikeness, and convexity  for harmonic mappings satisfying similar coefficient inequalities.

This paper studies the class $\mathcal{A}_b\label{symAb}$ consisting of functions $f(z)=z +\sum _{n=2}^{\infty}a_{n}z^{n}, (|a_2|=2b)$, $0\leq b\leq1$, in the disk $\mathbb{D}$. Univalent functions in $\mathcal{A}_b$ have been studied in \cite{Ali&Cho,sumit,Nagpal&Ravi,RaviRadii}. In \cite{RaviRadii}, Ravichandran obtained the sharp radii of starlikeness and convexity  of order $\alpha$ for  functions $f\in\mathcal{A}_b$ satisfying $|a_n|\leq n$ or $|a_n|\leq M$ ($M>0$), $n\geq3$.  The radius constants for uniform convexity and parabolic starlikeness for functions $f\in\mathcal{A}_b$ satisfying $|a_n|\leq n$, $n\geq3$ were also obtained.

In \cite{Lewandowski1}, Lewandowski \emph{et al.} proved that an analytic function $f$ satisfying
 \begin{equation}\label{Lewandoski}\RE\left(\frac{z^2f''(z)}{f(z)}+\frac{zf'(z)}{f(z)}\right)>0\quad(z\in\mathbb{D})\end{equation}
 is starlike. The class of such functions is easily extended to
 \begin{equation}\label{Ch5Lab}\RE\left(\alpha\frac{z^2f''(z)}{f(z)}+\frac{zf'(z)}{f(z)}\right)>\beta\quad(\alpha\geq0,\ \beta<1,\ z\in\mathbb{D}),\end{equation}
 and has subsequently been investigated in \cite{KSP, Nunokawa1, Obradovich, Padmanabhan1, Ravi1, Ravi2, liuZ}.   For $-\alpha/2\leq\beta<1 $, Li and Owa \cite{owa2} proved that  functions satisfying  \eqref{Ch5Lab} are starlike.

 Another related class is the class of analytic functions satisfying
 \[\RE\frac{zf'(z)}{f(z)}<\beta\quad(\beta>1,\ z\in\mathbb{D}). \]
 This class was studied in \cite{Nishiwaki, Uralegaddi94,Uralegaddi98, Owa&Srivastava}. In \cite{liu&liu}, Liu \emph{et al.} extended the class to functions satisfying
 \begin{equation}\label{H'(a,b)}\RE\left(\alpha\frac{z^2f''(z)}{f(z)}+\frac{zf'(z)}{f(z)}\right)<\beta\quad(\alpha\geq0,\ \beta>1,\ z\in\mathbb{D}).\end{equation}

 Now functions satisfying \eqref{Ch5Lab} or \eqref{H'(a,b)} evidently belongs to the class
\begin{equation}\label{R}
\mathcal{L}(\alpha, \beta):=\left\{f\in\mathcal{A}:\alpha\frac{z^2f''(z)}{f(z)}+\frac{zf'(z)}{f(z)}\prec \frac{1+(1-2\beta)z}{1-z},\ \beta\in\mathbb{R}\setminus\{1\},\ \alpha\geq0\right\}.\end{equation}
Denote by $\mathcal{L}_0(\alpha,\beta)$ its subclass consisting of  functions $f\in \mathcal{A}$ satisfying
\[\left|\alpha\frac{z^2f''(z)}{f(z)}+\frac{zf'(z)}{f(z)}-1\right|\leq|1-\beta| \quad(\beta\in\mathbb{R}\setminus\{1\},\ \alpha\geq0).\]

A sufficient  condition for functions $f\in\mathcal{A}$ to belong to
the class $\mathcal{L}(\alpha,\beta)$ is given in the following lemma.

\begin{lemma} \label{lem3.1}{\rm\cite{liuZ,Sun}}\label{lemmaR}  Let
$\beta\in\mathbb{R}\setminus\{1\}$, and $\alpha\geq0$.   If  $f(z)=z+\sum_{n=2}^{\infty}a_nz^n\in\mathcal{A}$
satisfies the inequality \begin{equation}\label{th3.2e1}
\sum_{n=2}^{\infty}\big(\alpha
n^2+(1-\alpha)n-\beta\big)|a_n|\leq|1-\beta|,
\end{equation} then
$f\in\mathcal{L}(\alpha, \beta)$.
 \end{lemma}

Next let $\mathcal{ST}[A,B]$ denote the class of Janowski starlike functions $f\in \mathcal{A}$ satisfying the subordination
\[\frac{zf'(z)}{f(z)}\prec\frac{1+Az}{1+Bz}\quad (-1\leq B<A\leq 1).\]
This class was introduced by Janowski \cite{Janowski}. Certain well-known subclasses of starlike functions are special cases of $\mathcal{ST}[A,B]$ for suitable choices of the parameters $A$ and $B$. For example, for  $0\leq\beta<1$, $\mathcal{ST}(\beta):=\mathcal{ST}[1-2\beta,-1]$ is
the familiar class of starlike functions of order $\beta$, and $\mathcal{ST}_{\beta}:=\mathcal{L}_0(0,\beta)=\mathcal{ST}[1-\beta,0]$. Janowski \cite{Janowski} obtained the exact value of the radius of convexity for $\mathcal{ST}[A,B]$.

Another result that will be required in  our investigation is the following result of Goel and Sohi \cite{Goel}.

\begin{lemma}{\rm\cite{Goel}}\label{LAB} Let $-1\leq B<A\leq 1$. If  $f(z)=z+\sum_{n=2}^{\infty}a_nz^n\in\mathcal{A}$ satisfies the inequality
\begin{equation}\label{EAB}\sum_{n=2}^{\infty}\big((1-B)n-(1-A)\big)|a_n|\leq A-B,\end{equation}
then $f\in \mathcal{ST}[A,B]$.
\end{lemma}

The Taylor coefficients of  functions $f(z)=z+\sum_{n=2}^{\infty}a_nz^n\in\mathcal{A}$ are known to satisfy certain coefficient inequality. For instance, starlike functions, convex functions in the direction of imaginary axis  and close-to-convex functions are bounded by $|a_n|\leq n\ (n\geq2)$ (\cite{Nevanlinna},  \cite[p.\ 210]{ch1Goodman}, \cite{Reade55}). Convex functions,  starlike functions of order 1/2, and starlike functions with respect to symmetric points  satisfy $|a_n|\leq 1\ (n\geq2)$ (\cite{Loewner1917},  \cite{Schild}, \cite{Sakaguchi}). Close-to-convex functions with argument $\beta$ satisfies $|a_n|\leq1+(n-1)\cos\beta$ \cite{GoodmanSaff}, while uniformly starlike functions are bounded by $|a_n|\leq 2/n\ (n\geq2)$ \cite{Goodman91'}, and the  uniformly convex functions by  $|a_n|\leq 1/n\ (n\geq2)$ \cite{Goodman91}. Simple examples show that the converse does not hold.

This paper studies functions $f=z+\sum_{n=2}^{\infty}a_nz^n\in\mathcal{A}_b$ satisfying either $|a_n|\leq cn+d$ ($c,\ d\geq0$) or $|a_n|\leq c/n$ ($c>0$) for $n\geq 3$. In the next section, sharp $\mathcal{L}(\alpha,\beta)$-radius and $\mathcal{ST}[A,B]$-radius are obtained for these classes. Several known radius constants are shown to be specific cases of the results obtained in this paper.

\section{Radius Constants}
First the sharp $\mathcal{L}(\alpha,\beta)$-radius  of   $f\in \mathcal{A}_b$ satisfying the coefficient inequality $|a_n|\leq cn+d$ is obtained in the following result.
\begin{theorem}\label{T2.1} Let $\beta\in\mathbb{R}\setminus\{1\}$, and $\alpha\geq0$. The  $\mathcal{L}(\alpha, \beta)$-radius of   $f(z)=z+\sum_{n=2}^{\infty}a_nz^n\in \mathcal{A}_b$ satisfying the coefficient inequality $|a_n|\leq cn+d$,  $c, d\geq0$ for $n\geq3$
is the real root  in $(0,1)$ of the equation
\begin{equation}\label{th1.e3}\begin{split} \big[(c+d)&(1-\beta)+|1-\beta|+(2\alpha+2-\beta)(2(c-b)+d)r\big](1-r)^4\\
&=c\alpha(1+4r+r^2)+\big((1-\alpha)c+\alpha d\big)(1-r^2)\\
& \quad{}+\big((1-\alpha)d-\beta c\big)(1-r)^2-\beta d(1-r)^3. \end{split}\end{equation}
 For $\beta<1$, this number is also the $\mathcal{L}_0 (\alpha,\beta)$-radius of $f\in \mathcal{A}_b$. The results are sharp.
\end{theorem}

\begin{proof}
 For $0\leq r_0<1$, the following identities hold:
\begin{align}
\sum_{n=2}^{\infty}r_0^{n}&=\frac{1}{1-r_0}-1-r_0\label{PS4},\\
\sum_{n=2}^{\infty}nr_0^{n}&=\frac{1}{(1-r_0)^2}-1-2r_0,\label{PS3}\\
\sum_{n=2}^{\infty}n^2r_0^{n}&=\frac{1+r_0}{(1-r_0)^3}-1-4r_0,\label{PS2}\\
\sum_{n=2}^{\infty}n^3r_0^{n}&=\frac{1+4r_0+r_0^2}{(1-r_0)^4}-1-8r_0.\label{PS1}
\end{align}
The number $r_0$ is the $\mathcal{L}(\alpha, \beta)$-radius of function $f\in \mathcal{A}_b$  if and only if $f(r_0z)/r_0\in\mathcal{L}(\alpha,\beta)$. Therefore, by Lemma \ref{lemmaR}, it is sufficient to verify  the inequality
\begin{equation}\label{1}
\sum_{n=2}^{\infty}\big(\alpha
n^2+(1-\alpha)n-\beta\big)|a_n|r_0^{n-1}\leq|1-\beta|,
\end{equation}where $r_0$ is the real root in $(0,1)$ of   \eqref{th1.e3}.
Using   \eqref{PS4}, \eqref{PS3}, \eqref{PS2}, and \eqref{PS1} for  $f\in\mathcal{A}_b$ lead to
\begin{align*}\lefteqn{\sum_{n=2}^{\infty}\big(\alpha n^2+(1-\alpha)n-\beta\big)|a_n|r_0^{n-1}}\\
&\leq2(2\alpha+2-\beta)br_0+\sum_{n=3}^{\infty}\big(\alpha n^2+(1-\alpha)n-\beta\big)(cn+d)r_0^{n-1}\\
&=2(2\alpha+2-\beta)br_0+c\alpha\left(\frac{1+4r_0+r_0^2}{(1-r_0)^4}-1-8r_0\right)\\
&\quad\quad +\big((1-\alpha)c+\alpha d\big)\left(\frac{1+r_0}{(1-r_0)^3}-1-4r_0\right)\\
&\quad\quad +\big((1-\alpha)d-\beta c\big)\left(\frac{1}{(1-r_0)^2}-1-2r_0\right)\\
&\quad\quad -\beta d\left(\frac{1}{1-r_0}-1-r_0\right)\\
&=(c+d)(\beta-1)-(2\alpha+2-\beta)\big(2(c-b)+d\big)r_0\\
&\quad\quad+\Big(c\alpha(1+4r_0+r_0^2)+\big((1-\alpha)c+\alpha d\big)(1-r_0^2)\\
&\quad\quad +\big((1-\alpha)d-\beta c\big)(1-r_0)^2-\beta d(1-r_0)^3\Big)/(1-r_0)^4\\
&=|1-\beta|.
\end{align*}

For $\beta<1$, consider the  function
 \begin{equation}\label{f0}f_0(z)=z-2bz^2-\sum_{n=3}^{\infty}(cn+d)z^n=(c+1)z+2(c-b)z^2-\frac{cz}{(1-z)^2}-\frac{dz^3}{1-z}.\end{equation}
At the point $z=r_0$ where $r_0$ is the  root in $(0,1)$ of \eqref{th1.e3},  $f_0$ satisfies
\begin{align}\label{middle}&\RE\left(\alpha\frac{z^2f''_0(z)}{f_0(z)}+\frac{zf_0'(z)}{f_0(z)}\right) =1-\frac{N(r_0)}
{D(r_0)}=\beta,
\end{align}
where \begin{align*}N(r_0)&=-2(c-b)(2\alpha+1)r_0+\dfrac{2cr_0(2\alpha+1)}{(1-r_0)^3}+\dfrac{6c\alpha r_0^2}{(1-r_0)^4}\\
&\quad\quad+\dfrac{2dr_0^2(3\alpha+1)}{1-r_0}+\dfrac{dr_0^3(6\alpha+1)}{(1-r_0)^2}+\dfrac{2dr_0^4\alpha}{(1-r_0)^3},\\
D(r_0)&=c+1+2(c-b)r_0-\dfrac{c}{(1-r_0)^2}-\dfrac{dr_0^2}{1-r_0}.\end{align*}
This shows that $r_0$ is the sharp  $\mathcal{L}(\alpha,\beta)$-radius for  $f\in\mathcal{A}_b$.
For $\beta<1$,  equation \eqref{middle} shows that the rational expression $N(r_0)/D(r_0)$  is positive, and therefore the following equality holds:
\begin{align*} &\left|\alpha\frac{z^2f''_0(z)}{f_0(z)}+\frac{zf_0'(z)}{f_0(z)}-1\right|=1-\beta.\end{align*}
Thus, $r_0$ is  the sharp $\mathcal{L}_0(\alpha,\beta)$-radius for   $f\in\mathcal{A}_b$ when $\beta<1$.

For $\beta>1$,
 \begin{equation}f_0(z)=z+2bz^2+\sum_{n=3}^{\infty}(cn+d)z^n=(1-c)z+2(b-c)z^2+\frac{cz}{(1-z)^2}+\frac{dz^3}{1-z}\end{equation}
shows sharpness of the result. The proof is similar to the case $\beta<1$, and is thus omitted.
\end{proof}

\begin{theorem} \label{T2'}Let $\beta\in\mathbb{R}\setminus\{1\}$, and $\alpha\geq0$. The $\mathcal{L} (\alpha,\beta)$-radius of  $f(z)=z+\sum_{n=2}^{\infty}a_nz^n\in \mathcal{A}_b$ satisfying the coefficient inequality $|a_n|\leq c/n$ for $n\geq3$ and $c>0$ is the real root in $(0,1)$ of
\begin{equation} \label{th2.e1} \begin{split}\lefteqn{\left[c(1-\beta)+|1-\beta|+(2\alpha+2-\beta)r\left(\frac{c}{2}-2b\right)\right](1-r)^2}\\
&=c\alpha +(1-\alpha)c(1-r)+\beta c(1-r)^2\frac{\log(1-r)}{r}.\end{split}\end{equation}
For $\beta<1$, this number is also the $\mathcal{L}_0 (\alpha,\beta)$-radius of   $f\in \mathcal{A}_b$. The results are sharp.
\end{theorem}

\begin{proof} By Lemma \ref{lemmaR}, $r_0$ is the $\mathcal{L} (\alpha,\beta)$-radius of functions $f\in\mathcal{A}_b$ when the inequality  \eqref{1} holds for  the real root $r_0$ of   equation \eqref{th2.e1} in $(0,1)$. Using \eqref{PS4} and \eqref{PS3} together with
\begin{align}
 \sum_{n=2}^{\infty}\frac{r_0^{n}}{n}&=-\frac{\log(1-r_0)}{r_0}-1-\frac{r_0}{2},\label{PS5}
\end{align}
   for $f\in\mathcal{A}_b$ imply that
\begin{align*}\lefteqn{\sum_{n=2}^{\infty}\big(\alpha n^2+(1-\alpha)n-\beta\big)|a_n|r_0^{n-1}}\\
&\leq2(2\alpha+2-\beta)br_0+\sum_{n=3}^{\infty}\big(\alpha n^2+(1-\alpha)n-\beta\big)\left(\frac{c}{n}\right)r_0^{n-1}\\
&=2(2\alpha+2-\beta)br_0+c\alpha\left(\frac{1}{(1-r_0)^2}-1-2r_0\right)\\
&\quad\quad+(1-\alpha)c\left(\frac{1}{1-r_0}-1-r_0\right)\\
&\quad\quad-\beta c\left(-\frac{\log(1-r_0)}{r_0}-1-\frac{r_0}{2}\right)\\
&=c(\beta-1)+(2\alpha+2-\beta)r_0\left(2b-\frac{c}{2}\right)\\
&\quad\quad+\frac{c\alpha r_0+(1-\alpha)c(1-r_0)r_0+\beta c(1-r_0)^2\log(1-r_0)}{(1-r_0)^2r_0}\\
&=|1-\beta|.
\end{align*}

  To verify sharpness for $\beta<1$, consider the function
\begin{equation}\label{f0'}f_0(z)=z-2bz^2-\sum_{n=3}^{\infty}\frac{c}{n}z^n=(1+c)z+\left(\frac{c}{2}-2b\right)z^2+c\log(1-z).
\end{equation}
At the point $z=r_0$ where $r_0$ is the  root in $(0,1)$ of  equation  \eqref{th2.e1}, $f_0$ satisfies
\begin{align}\label{middle2}\lefteqn{ \RE\left(\alpha\frac{z^2f''_0(z)}{f_0(z)}+\frac{zf_0'(z)}{f_0(z)}\right)}\notag \\
&=1-\frac{-\left(\dfrac{c}{2}-2b\right)r_0(2\alpha+1)+\dfrac{cr_0\alpha}{(1-r_0)^2}+\dfrac{c}{1-r_0}+\dfrac{c\log(1-r_0)}{r_0}}{(1+c)+\left(\dfrac{c}{2}-2b\right)r_0
+\dfrac{c\log(1-r_0)}{r_0}}\\
&=\beta.\notag
\end{align}
Thus  $r_0$  is the sharp $\mathcal{L} (\alpha,\beta)$-radius of  $f\in\mathcal{A}_b$.
Since $\beta<1$,  the rational expression in \eqref{middle2} is positive, and therefore
\begin{align*} &\left|\alpha\frac{z^2f''_0(z)}{f_0(z)}+\frac{zf_0'(z)}{f_0(z)}-1\right|=1-\beta,
\end{align*}
which shows   $r_0$ is the sharp $\mathcal{L}_0 (\alpha,\beta)$-radius of  $f\in\mathcal{A}_b$. For $\beta>1$, the sharpness of the result is demonstrated by the function $f_0$ given by
\begin{align*}f_0(z)&=z+2bz^2+\sum_{n=3}^{\infty}\frac{c}{n}z^n\\
&=(1-c)z+\left(2b-\frac{c}{2}\right)z^2-c\log(1-z).\qedhere\end{align*}
\end{proof}

\begin{remark}
\begin{enumerate}\item[]
\item For  $\alpha=0$, $\beta=0$, $c=1$, $d=0$, and $0 \leq b \leq1$, Theorem \ref{T2.1} yields the radius of starlikeness obtained by Yamashita \cite{Yamashita}.
\item For  $\alpha=0$, $c=1$, and $d=0$,  Theorem \ref{T2.1} reduces to Theorem 2.1 in \cite[p. 3]{RaviRadii}. When $\alpha=0$, $c=0$, and $d=M$,
Theorem \ref{T2.1} leads to Theorem 2.5 in \cite[p. 5]{RaviRadii}.
\item For $\alpha=0$, Theorem \ref{T2'} yields  the radius of starlikeness of order $\beta$ for $f\in \mathcal{A}_b$
 obtained by Ravichandran   \cite[Theorem 2.8]{RaviRadii}.

\end{enumerate}
\end{remark}
%
%

The next result finds the sharp $\mathcal{ST}[A,B]$-radius of  $f\in \mathcal{A}_b$ satisfying the coefficient inequality $|a_n|\leq cn+d$,  $c,d\geq0$ for $n\geq3$.
\begin{theorem} \label{TAB}Let $-1\leq B<A\leq 1$. The $\mathcal{ST}[A,B]$-radius of  $f(z)=z+\sum_{n=2}^{\infty}a_nz^n\in \mathcal{A}_b$ satisfying the coefficient inequality $|a_n|\leq cn+d$ for $n\geq3$ and $c, d\geq0$ is the real root in $(0,1)$ of
\begin{equation}\label{th3.e3}\begin{split}\lefteqn{[(A-B)(c+d+1) -(2b-2c-d)\big(2(1-B)-(1-A)\big)r](1-r)^3}\\
&=c(1-B)(1+r)+\big(d(1-B)-c(1-A)\big)(1-r)-(1-A)d(1-r)^2.\end{split}\end{equation}
The result is sharp.
\end{theorem}

\begin{proof}It is evident  that $r_0$ is the $\mathcal{ST}[A,B]$-radius of  $f\in \mathcal{A}_b$ if and only if $f(r_0z)/r_0\in\mathcal{ST}[A,B]$. Hence, by Lemma \ref{LAB}, it suffices to show
\begin{equation}\label{AB}\sum_{n=2}^{\infty}\big((1-B)n-(1-A)\big)|a_n|r_0^{n-1}\leq A-B\quad (-1\leq B<A\leq 1),\end{equation}
where $r_0$ is the root in $(0,1)$ of  equation \eqref{th3.e3}. From  \eqref{PS4}, \eqref{PS3}, and \eqref{PS2} for  function $f\in\mathcal{A}_b$,   it follows that
\begin{align*}\lefteqn{\sum_{n=2}^{\infty}\big((1-B)n-(1-A)\big)|a_n|r_0^{n-1}}\\
&\leq 2\big(2(1-B)-(1-A)\big)br_0+\sum_{n=3}^{\infty}\big((1-B)n-(1-A)\big)(cn+d)r_0^{n-1}\\
&= 2\big(2(1-B)-(1-A)\big)br_0+c(1-B)\left(\frac{1+r_0}{(1-r_0)^3}-1-4r_0\right)\\
&\quad\quad+\big(d(1-B)-c(1-A)\big)\left(\frac{1}{(1-r_0)^2}-1-2r_0\right)\\
&\quad\quad-(1-A)d\left(\frac{1}{1-r_0}-1-r_0\right)\\
&=(B-A)(c+d)+(2b-2c-d)\big(2(1-B)-(1-A)\big)r_0\\
&\quad\quad+\Big(c(1-B)(1+r_0)+\big(d(1-B)-c(1-A)\big)(1-r_0)\\
&\quad\quad-(1-A)d(1-r_0)^2\Big)/(1-r_0)^3\\
&=A-B.\end{align*}

The function  $f_0$  given by \eqref{f0} shows that the result is sharp. Indeed, at the point $z=r_0$ where $r_0$ is the root in $(0,1)$ of  equation \eqref{th3.e3}, the function $f_0$ satisfies
  \begin{align*}
 \left|\frac{zf'_0(z)}{f_0(z)}-1\right|=\frac{-2(c-b)r_0+\dfrac{2dr_0^2}{1-r_0}+\dfrac{dr_0^3}{(1-r_0)^2}+\dfrac{2cr_0}{(1-r_0)^3}}
 {c+1+2(c-b)r_0-\dfrac{c}{(1-r_0)^2}-\dfrac{dr_0^2}{1-r_0}},&
 \end{align*} and
 \begin{align*}\left|A-B \frac{zf'_0(z)}{f_0(z)}\right|&=\frac{(c+1)(A-B)+2(c-b)r_0(A-2B)}{c+1+2(c-b)r_0-\dfrac{c}{(1-r_0)^2}-\dfrac{dr_0^2}{1-r_0}}\\
 &\quad\quad-\frac{\dfrac{c(A-B)}{(1-r_0)^2}+\dfrac{2cr_0B}{(1-r_0)^3}-\dfrac{dr_0^2(A-3B)}{1-r_0}+\dfrac{dr_0^3B}{(1-r_0)^2}}
 {c+1+2(c-b)r_0-\dfrac{c}{(1-r_0)^2}-\dfrac{dr_0^2}{1-r_0}}.
 \end{align*}
 Then   \eqref{th3.e3} yields
  \begin{equation}\label{th3.e1}\left|\frac{zf_0'(z)}{f_0(z)}-1\right|=\left|A-B\frac{zf_0'(z)}{f_0(z)}\right| \quad(-1\leq B<A\leq 1,\ z=r_0), \end{equation}
or equivalently $f_0\in\mathcal{ST}[A,B]$.
\end{proof}

\begin{theorem}\label{T2''} Let $-1\leq B<A\leq 1$. The $\mathcal{ST}[A,B]$-radius of   $f(z)=z+\sum_{n=2}^{\infty}a_nz^n\in \mathcal{A}_b$ satisfying the coefficient inequality $|a_n|\leq c/n$ for $n\geq3$ and $c>0$ is the real root in $(0,1)$ of the equation
\begin{equation}\label{th4.e3} \begin{split}&\left((c+1)(A-B)-\big(2(1-B)-(1-A)\big)r\left(2b-\frac{c}{2}\right)\right)(1-r)\\
&\quad =c(1-B)+c(1-A)(1-r)\frac{\log(1-r)}{r} \end{split}\end{equation}
The result is sharp.
\end{theorem}

\begin{proof} By Lemma \ref{LAB}, condition \eqref{AB} assures that $r_0$ is  the $\mathcal{ST}[A,B]$-radius of   $f\in \mathcal{A}_b$ where $r_0$ is the real root of  \eqref{th4.e3}. Therefore, using
\eqref{PS4} and \eqref{PS5}  for  $f\in\mathcal{A}_b$ yield
\begin{align*}\lefteqn{\sum_{n=2}^{\infty}\big((1-B)n-(1-A)\big)|a_n|r_0^{n-1}}\\
&\leq 2\big(2(1-B)-(1-A)\big)br_0+\sum_{n=3}^{\infty}\big((1-B)n-(1-A)\big)\left(\frac{c}{n}\right)r_0^{n-1}\\
&=2\big(2(1-B)-(1-A)\big)br_0+c(1-B)\left(\frac{1}{1-r_0}-1-r_0\right)\\
&\quad\quad-c(1-A)\left(-\frac{\log(1-r_0)}{r_0}-1-\frac{r_0}{2}\right)\\
&=c(B-A)+\big(2(1-B)-(1-A)\big)r_0\left(2b-\frac{c}{2}\right)\\
&\quad\quad+\frac{c(1-B)r_0+c(1-A)(1-r_0)\log(1-r_0)}{(1-r_0)r_0}\\
&=A-B.
\end{align*}
 The result is sharp for the function  $f_0$  given  by \eqref{f0'}. Indeed,
$f_0$ satisfies
  \begin{align*}
\left| \frac{zf'_0(z)}{f_0(z)}-1\right|=\frac{-\left(\dfrac{c}{2}-2b\right)r_0+\dfrac{c}{1-r_0}+\dfrac{c\log(1-r_0)}{r_0}}{(1+c)
+\left(\dfrac{c}{2}-2b\right)r_0+\dfrac{c\log(1-r_0)}{r_0}},
 \end{align*}
and
 \begin{align*}
\left|A-B \frac{zf'_0(z)}{f_0(z)}\right|=\frac{(1+c)(A-B)+(A-2B)\left(\dfrac{c}{2}-2b\right)r_0+\dfrac{cB}{1-r_0}+\dfrac{cA\log(1-r_0)}{r_0}}{(1+c)
+\left(\dfrac{c}{2}-2b\right)r_0
+\dfrac{c\log(1-r_0)}{r_0}},
 \end{align*}
 at the point $z=r_0$ where $r_0$ is the  root in $(0,1)$ of  equation \eqref{th4.e3}. From \eqref{th4.e3}, the function $f_0$ is seen to satisfy \eqref{th3.e1}, and hence the result is sharp.
\end{proof}

%

\end{document}